\documentclass[a4paper,12pt]{article}
\usepackage{a4wide}
\usepackage{amsmath}
\usepackage{amssymb}
\usepackage{amsthm}
\usepackage{latexsym}
\usepackage{graphicx}
\usepackage[english]{babel}
\usepackage{makeidx}

\newtheorem{prop}[subsection]{Proposition}

\newtheorem{teor}[subsection]{Theorem}
\newtheorem*{teor*}{Theorem}
\newtheorem{lema}[subsection]{Lemma}

\newcommand{\pa}{p_{\mathbf a}}

\newcommand{\za}{\zeta_{\mathbf a}}

\newcommand{\alf}{\underline{\alpha}}

\def\Ree{\operatorname{Re}}

\begin{document}
\selectlanguage{english}
\frenchspacing
\numberwithin{equation}{section}

\large
\begin{center}
\textbf{On the restricted partition function via determinants with Bernoulli polynomials. II}

Mircea Cimpoea\c s
\end{center}
\normalsize

\begin{abstract}
Let $r\geq 1$ be an integer, $\mathbf a=(a_1,\ldots,a_r)$ a vector of positive integers and let $D\geq 1$ be a common multiple of $a_1,\ldots,a_r$.
We  prove that, if $D=1$ or $D$ is a prime number then the restricted partition function $\pa(n): = $ the number 
of integer solutions $(x_1,\dots,x_r)$ to $\sum_{j=1}^r a_jx_j=n$ with $x_1\geq 0, \ldots, x_r\geq 0$ can be computed by solving a system of linear equations with 
coefficients which are values of Bernoulli polynomials and Bernoulli Barnes numbers.
 

\end{abstract}

\section{Introduction}

Let $\mathbf a:=(a_1,a_2,\ldots,a_r)$ be a sequence of positive integers, $r\geq 1$. 
The \emph{restricted partition function} associated to $\mathbf a$ is $\pa:\mathbb N \rightarrow \mathbb N$, 
$\pa(n):=$ the number of integer solutions $(x_1,\ldots,x_r)$ of $\sum_{i=1}^r a_ix_i=n$ with $x_i\geq 0$.
Let $D$ be a common multiple of $a_1,\ldots,a_r$. According to \cite{bell}, $\pa(n)$ is a 
quasi-polynomial of degree $r-1$, with the period $D$, i.e.
\begin{equation}\label{pan}
\pa (n) = d_{\mathbf a,r-1}(n)n^{r-1} + \cdots + d_{\mathbf a,1}(n)n + d_{\mathbf a,0}(n),\;(\forall)n\geq 0,
\end{equation}
where $d_{\mathbf a,m}(n+D)=d_{\mathbf a,m}(n)$, $(\forall)0\leq m\leq r-1,n\geq 0$, and $d_{\mathbf a,r-1}(n)$ is not identically zero.
The restricted partition function $\pa(n)$ was studied extensively in the literature, starting with the works of Sylvester \cite{sylvester} and Bell \cite{bell}.
Popoviciu \cite{popoviciu} gave a precise formula for $r=2$. Recently, Bayad and Beck \cite[Theorem 3.1]{babeck} proved an
explicit expression of $\pa(n)$ in terms of Bernoulli-Barnes polynomials and the Fourier Dedekind sums, in the case that
$a_1,\ldots,a_r$ are are pairwise coprime. In \cite{lucrare} we proved that the computation of $\pa(n)$ can be reduced to solving
the linear congruency $a_1j_1+\cdots+a_rj_r\equiv n(\bmod\;D)$ in the range $0\leq j_1\leq \frac{D}{a_1},\ldots,0\leq j_r\leq \frac{D}{a_r}$. 
In \cite{luc} we proved that if a determinant $\Delta_{r,D}$, which depends only on $r$ and $D$, with entries consisting in values of Bernoulli polynomials
is nonzero, then $\pa(n)$ can be computed in terms of values of Bernoulli polynomials and Bernoulli Barnes numbers. The aim of our paper is to tackle
the same problem, from another perspective which relays on the arithmetic properties of Bernoulli polynomials. 

First we recall some definitions.
The \emph{Barnes zeta} function associated to $\mathbf a$ and $w>0$ is
$$ \za(s,w):=\sum_{n=0}^{\infty} \frac{\pa(n)}{(n+w)^s},\; \Ree s>r,$$
see \cite{barnes} and \cite{spreafico} for further details. It is well known that $\za(s,w)$ is meromorphic on $\mathbb C$ with poles at most in the set $\{1.\ldots,r\}$.
We consider the function
\begin{equation}
\za(s) := \lim_{w\searrow 0}(\za(s,w)-w^{-s}).
\end{equation}
In \cite[Lemma 2.6]{lucrare} we proved that 
\begin{equation}
\za(s)=\frac{1}{D^s}\sum_{m=0}^{r-1}\sum_{v=1}^{D} d_{\mathbf a,m}(v)D^m \zeta(s-m,\frac{v}{D}), 
\end{equation}
where $$\zeta(s,w):=\sum_{n=0}^{\infty} \frac{1}{(n+w)^s},\;\Ree s>1,$$
is the \emph{Hurwitz zeta} function. See also \cite{cori}.
The \emph{Bernoulli numbers} $B_j$ are defined by
$$ \frac{z}{e^z-1} = \sum_{j=0}^{\infty}B_j \frac{z^j}{j!}, $$
$B_0=1$, $B_1=-\frac{1}{2}$, $B_2=\frac{1}{6}$, $B_4=-\frac{1}{30}$ and $B_n=0$ if $n$ is odd and greater than $1$.
The \emph{Bernoulli polynomials} are defined by
$$\frac{ze^{xz}}{(e^z-1)}=\sum_{n=0}^{\infty}B_n(x)\frac{z^n}{n!}.$$
They are related with the Bernoulli numbers by
\begin{equation}\label{berp}
B_n(x)=\sum_{k=0}^n \binom{n}{k}B_{n-k}x^k. 
\end{equation}
It is well know, see for instance \cite[Theorem 12.13]{apostol}, that
\begin{equation}
\zeta(-n,w)=-\frac{B_{n+1}(w)}{n+1},\;(\forall)n\in \mathbb N,w>0. 
\end{equation}
The \emph{Bernoulli-Barnes polynomials} are defined by
$$\frac{z^r e^{xz}}{(e^{a_1 z}-1)\cdots (e^{a_r z}-1)}=\sum_{j=0}^{\infty}B_j(x;\mathbf a)\frac{z^j}{j!}. $$
The \emph{Bernoulli-Barnes numbers} are defined by
$$B_j(\mathbf a):=B_j(0;\mathbf a)=\sum_{i_1+\cdots+ i_r=j}\binom{j}{i_1,\ldots,i_r} B_{i_1}\cdots B_{i_r}a_1^{i_1-1}\cdots a_r^{i_r-1}.$$
According to \cite[Formula (3.10)]{rui}, it holds that
\begin{equation}
\za(-n,w)=\frac{(-1)^r n!}{(n+r)!}B_{r+n}(w;\mathbf a),\;(\forall)n\in\mathbb N. 
\end{equation}
From $(1.2)$ and $(1.6)$ it follows that
\begin{equation}
\za(-n)=\frac{(-1)^r n!}{(n+r)!}B_{r+n}(\mathbf a),\;(\forall)n\geq 1.
\end{equation}
From $(1.3),(1.5)$ and $(1.7)$ it follows that
\begin{equation}\label{poola}
\sum_{m=0}^{r-1}\sum_{v=1}^{D} d_{\mathbf a,m}(v)D^{n+m} \frac{B_{n+m+1}(\frac{v}{D})}{n+m+1}  = \frac{(-1)^{r-1} n!}{(n+r)!}B_{r+n}(\mathbf a),\;(\forall)n\geq 1,
\end{equation}
Let $\alf: \alpha_1 < \alpha_2 < \cdots < \alpha_{rD}$ be a sequence of integers with $\alpha_1\geq 2$. Substituting $n$ with $\alpha_j-1$, $1\leq j\leq rD$, in \eqref{poola}
and multiplying with $D$, we obtain the  system of linear equations 
\begin{equation}\label{poolah}
\sum_{m=0}^{r-1}\sum_{v=1}^{D} d_{\mathbf a,m}(v) \frac{D^{\alpha_j+m} B_{\alpha_j+m}(\frac{v}{D})}{\alpha_j+m}  = \frac{(-1)^{r-1} (\alpha_j-1)!D}{(\alpha_j+r-1)!}B_{\alpha_j+r-1}(\mathbf a),\;1\leq j\leq rD,
\end{equation}
which has the determinant \small{
\begin{equation}\label{pista}
\Delta_{r,D}(\alf):= \begin{vmatrix}
\frac{D^{\alpha_1}B_{\alpha_1}(\frac{1}{D})}{\alpha_1} & \cdots & \frac{D^{\alpha_1}B_{\alpha_1}(1)}{\alpha_1} 
& \cdots & \frac{D^{\alpha_1}B_{\alpha_1+r-1}(\frac{1}{D})}{\alpha_1+r-1} & \cdots & \frac{D^{\alpha_1}B_{\alpha_1+r-1}(1)}{\alpha_1+r-1} \\
\frac{D^{\alpha_2}B_{\alpha_2}(\frac{1}{D})}{\alpha_2} & \cdots & \frac{D^{\alpha_2}B_{\alpha_2}(1)}{\alpha_2} 
& \cdots & \frac{D^{\alpha_2}B_{\alpha_2+r-1}(\frac{1}{D})}{\alpha_2+r-1} & \cdots & \frac{D^{\alpha_2}B_{\alpha_2+r-1}(1)}{\alpha_2+r-1} \\
\vdots & \vdots & \vdots & \vdots & \vdots & \vdots & \vdots  \\
\frac{D^{\alpha_{rD}}B_{\alpha_{rD}}(\frac{1}{D})}{\alpha_{rD}} & \cdots & \frac{\tilde B_{\alpha_{rD}}(\frac{D-1}{D})}{\alpha_{rD}} 
& \cdots & \frac{D^{\alpha_{rD}}B_{\alpha_{rD}+r-1}(\frac{1}{D})}{\alpha_{rD}+r-1} & \cdots & \frac{D^{\alpha_{rD}}B_{\alpha_{rD}+r-1}(1)}{\alpha_{rD}+r-1} 
\end{vmatrix} \end{equation}}

Note that, with the notation given in \cite[(2.10)]{luc}, we have that $\Delta_{r,D}=\Delta_{r,D}(0,1,\ldots,rD-1)$, here ommiting the condition $\alpha_1\geq 2$.

\begin{prop}
With the above notations, if $\Delta_{r,D}(\alf)\neq 0$, then 
$$d_{\mathbf a,m}(v) = \frac{\Delta_{r,D}^{m,v}(\alf)}{\Delta_{r,D}(\alf)},\;(\forall) 1\leq v\leq D, 0\leq m\leq r-1,$$ 
where $\Delta_{r,D}^{m,v}(\alf)$ is the determinant obtained from $\Delta_{r,D}(\alf)$, as defined in $(\ref{pista})$, by replacing the 
$(mD+v)$-th column with the column $(\frac{(-1)^{r-1} (\alpha_j-1)!D}{(\alpha_j+r-1)!}B_{\alpha_j+r-1}(\mathbf a))_{1\leq j\leq rD-1}$. Moreover,
$$\pa(n)=\frac{1}{\Delta_{r,D}(\alf)}\sum_{m=0}^{r-1} \Delta_{r,D}^{m,v}(\alf) n^m,\;(\forall)n\in\mathbb N.$$
\end{prop}

\begin{proof}
It follows from \eqref{poola} and \eqref{pista} by Cramer's rule.
The last assertion follows from \eqref{pan}.
\end{proof}

Our main theorem is the following:

\begin{teor}
Let $r\geq 1$ and let $D=1$ or $D\geq 2$ is a prime number. There exists a sequence of integers $\alf:\alpha_1<\alpha_2<\cdots<\alpha_{rD}$, $\alpha_1\geq 2$, such that 
 $\Delta_{r,D}(\alf)\neq 0$. In particular, we can compute $\pa(n)$ in terms of values of Bernoulli polynomials and Bernoulli-Barnes numbers.
\end{teor}

We believe that the result holds for any integer $D\geq 1$. Unfortunately, our method based on p-adic value and congruences for Bernoulli numbers and for values of Bernoulli polynomials,
is not refined enough to prove it.

\newpage
\section{Properties of Bernoulli polynomials}

We recall several properties of the Bernoulli polynomials. We have that:
\begin{equation}\label{unu}
 B_n(1-x)=(-1)^nB_n(x),\;(\forall)x\in\mathbb R,\;n\in\mathbb N.
\end{equation}
For any integers $n\geq 1$ and $1\leq v\leq D$, using \eqref{berp}, we let
\begin{equation}\label{berpe}
\tilde B_n(x):=D^n(B_n(x)-B_n) = \sum_{j=1}^{n-1} \binom{n}{j}D^j(xD)^{n-j}
\end{equation}
According to \cite[Theorem 1]{alm}, it holds that
\begin{equation}\label{int}
\tilde B_n(\frac{v}{D}) \in \mathbb Z,\;(\forall)1\leq v\leq D.
\end{equation}
According to the a result of T. Clausen and C. von Staudt (see \cite{clausen},\cite{staudt}), we have that
\begin{equation}\label{hopa}
B_{2n}=A_{2n} - \sum_{p-1|2n}\frac{1}{p},\;(\forall)n\geq 1,
\end{equation}
where $A_{2n}\in\mathbb Z$ and the sum is over the all primes $p$ such that $p - 1|2n$.

Let $p$ be a prime. For any integer $a$, the $p$-adic order of $a$ is $v_p(a):=\max\{k\;:\;p^k|a\}$, if $a\neq 0$,
and $v_p(0)=\infty$. For $q=\frac{a}{b}\in\mathbb Q$, the $p$-adic order of $q$ is $v_p(q):=v_p(a)-v_p(b)$. 
Note that \eqref{hopa} implies
\begin{equation}\label{hopaa}
 v_p(B_{2n}) = \begin{cases} -1,& p-1|2n \\ \geq 0,& p-1\nmid 2n \end{cases}.
\end{equation}

\begin{lema}
For any integer $n\geq 1$, it holds that:
\begin{enumerate}
 \item[(1)] $\tilde B_n(\frac{1}{2})=0$, if $n$ is odd, and $\tilde B_n(\frac{1}{2})\equiv 1(\bmod\; 2)$, if $n$ is even.
 \item[(2)] If $p$ is a prime, then $\tilde B_n(\frac{v}{p})\equiv v^n (\bmod\; p)$, $(\forall)1\leq v\leq p-1$.
\end{enumerate}
\end{lema}

\begin{proof}
(1) From \eqref{unu} it follows that $B_n(\frac{1}{2})=0$ if $n$ is odd, hence, as $B_n=0$, we get
$$\tilde B_n(\frac{1}{2})=D^n(B_n(\frac{1}{2})-B_n)=0.$$
Assume $n$ is even. According to \eqref{berpe}, we have that 
$$\tilde B_n(\frac{1}{2})=\sum_{j=0}^n \binom{n}{j}B_j 2^j.$$
Since $2|2nB_1 = -n$ and $v_2(2^jB_j)\geq 1$ for any $j\geq 2$, the conclusion follows immediately.

(2) According to \eqref{berpe}, we have that 
$$\tilde B_n(\frac{a}{p})=\sum_{j=0}^n \binom{n}{j}B_j v^{n-j}p^j.$$
From \eqref{hopaa}, we have that $v_p(p^jB_j)\geq 1$ for $j\geq 1$, hence the conclusion follows immediately.
\end{proof}

\begin{lema}
 If $p$ is a prime such that $p\nmid D$ then $\tilde B_p(\frac{v}{D})\equiv 0 (\bmod\; p)$, $(\forall)1\leq v\leq D-1$.
\end{lema}

\begin{proof}
 We have that $$\tilde B_p(\frac{v}{D})=\sum_{j=0}^p \binom{p}{j} B_j v^{p-j}D^j.$$
 Since $v_p(B_j)\geq 0$ for $j\leq p-2$, it follows that
\begin{equation}
v_p(\binom{p}{j}B_j)\geq 1,\;(\forall)1\leq j\leq p-2.
\end{equation}
On the other hand, from \eqref{hopa}, it follows that
\begin{equation}
v^p + \binom{p}{p-1}B_{p-1}vD^{p-1} \equiv v^p - vD^{p-1} \equiv v^p - v \equiv 0 (\bmod p), 
\end{equation}
hence we get the required result.
\end{proof}

\section{Preliminary results}

\begin{prop}(Case $D=1$)
Let $p_1<p_2<\cdots<p_r$ be some primes such that $p_1>2$ and $p_{j+1}-p_j>r$, $(\forall)1\leq j\leq r$. Let 
$\alpha_j:=p_j-j$, $1 \leq j \leq r$. We have that $\Delta_{r,1}(\alf)\neq 0$.
\end{prop}

\begin{proof}
Note that \eqref{unu} implies $B_n(1)=B_n$ for any $n\geq 2$. It follows that
\begin{equation}\label{trolo}
\Delta_{r,1}(\alf)=\begin{vmatrix}
\frac{B_{\alpha_1}}{\alpha_1} & \frac{B_{\alpha_1+1}}{\alpha_1+1} & \cdots & \frac{B_{\alpha_1+r-1}}{\alpha_1+r-1} \\
\frac{B_{\alpha_2}}{\alpha_2} & \frac{B_{\alpha_1+1}}{\alpha_2+1} & \cdots & \frac{B_{\alpha_2+r-1}}{\alpha_2+r-1} \\
\vdots & \vdots & \vdots & \vdots \\
\frac{B_{\alpha_r}}{\alpha_r} & \frac{B_{\alpha_1+1}}{\alpha_r+1} & \cdots & \frac{B_{\alpha_r+r-1}}{\alpha_r+r-1} \\
\end{vmatrix}.
\end{equation}
From \eqref{hopa} it follows that $v_{p_j}(B_{\alpha_j+j-1})=-1$ and $v_{p_j}(B_{\alpha_j+k-1})\geq 0$ for and $1 \leq k\leq r$ with $k\neq j$.
Moreover, if $1\leq \ell < j \leq r$, then, by hypothesis, $v_{p_j}(B_{\alpha_{\ell}+k-1})\geq 0$ for any $1\leq k\leq r$ (We implicitly used 
the fact that $B_n=0$ if $n\geq 3$ is odd). It follows that,
in the expansion of $\Delta_{r,1}(\alf)$ written in \eqref{trolo}, the term
$$\prod_{j=1}^r \frac{D^{\alpha_j+j-1}B_{\alpha_j+j-1}}{\alpha_j+j-1}$$
can not be simplified, hence $\Delta_{r,1}(\alf)\neq 0$.
\end{proof}

In the following, we assume $D\geq 2$ and we consider the determinant
\begin{equation}\label{pistaa}
\tilde{\Delta}_{r,D}(\alf):= \begin{vmatrix}
\frac{\tilde B_{\alpha_1}(\frac{1}{D})}{\alpha_1} & \cdots & \frac{\tilde B_{\alpha_1}(\frac{D-1}{D})}{\alpha_1} 
& \cdots & \frac{\tilde B_{\alpha_1+r-1}(\frac{1}{D})}{\alpha_1+r-1} & \cdots & \frac{\tilde B_{\alpha_1+r-1}(\frac{D-1}{D})}{\alpha_1+r-1} \\
\frac{\tilde B_{\alpha_2}(\frac{1}{D})}{\alpha_2} & \cdots & \frac{\tilde B_{\alpha_2}(\frac{D-1}{D})}{\alpha_2} 
& \cdots & \frac{\tilde B_{\alpha_2+r-1}(\frac{1}{D})}{\alpha_2+r-1} & \cdots & \frac{\tilde B_{\alpha_2+r-1}(\frac{D-1}{D})}{\alpha_2+r-1} \\
\vdots & \vdots & \vdots & \vdots & \vdots & \vdots & \vdots  \\
\frac{\tilde B_{\alpha_{rD-r}}(\frac{1}{D})}{\alpha_{rD-r}} & \cdots & \frac{\tilde B_{\alpha_{rD-r}}(\frac{D-1}{D})}{\alpha_{rD-r}} 
& \cdots & \frac{\tilde B_{\alpha_{rD-r}+r-1}(\frac{1}{D})}{\alpha_{rD-r}+r-1} & \cdots & \frac{\tilde B_{\alpha_{rD-r}+r-1}(\frac{D-1}{D})}{\alpha_{rD-r}+r-1} 
\end{vmatrix} \end{equation}
Let $p_1<p_2<\ldots<p_r$ be some primes such that $$p_1\geq \alpha_{r(D-1)}+r\text{ and }p_{j+1}-p_j>r, (\forall)1\leq j\leq r-1.$$ 
We let
\begin{equation}\label{stuf}
\alpha_{rD-r+j}:=p_j-j,\; (\forall)1\leq j\leq r. 
\end{equation}
According to Lemma $2.2$ and \eqref{stuf}, we have that
\begin{equation}\label{stufu}
 v_{p_{\ell}}\left(\frac{\tilde B_{\alpha_{rD-r+j}+j}(\frac{v}{D})}{\alpha_{rD-r+j}+j}\right)\geq 0,\;(\forall)1\leq j,\ell \leq r,\; 1\leq v\leq D-1.
\end{equation}
On the other hand, since $p_j\geq \alpha_{r(D-1)}+r$, from Lemma $2.2$ it follows that
\begin{equation}\label{stuff}
 v_{p_{\ell}}\left(\frac{D^{\alpha_t+j}B_{\alpha_{t}+j}}{\alpha_{t}+j}\right)\geq 0,\;
 v_{p_{\ell}}\left(\frac{\tilde B_{\alpha_{t}+j}(\frac{v}{D})}{\alpha_{t}+j}\right)\geq 0,\;(\forall)1\leq j,\ell \leq r,\;1\leq t\leq r(D-1),\; 1\leq v\leq D-1.
\end{equation}
Also, from \eqref{hopaa} and \eqref{stuf}, it follows that
\begin{equation}\label{stufy}
v_{p_{\ell}}\left(\frac{B_{\alpha_{rD-r+j}+j}(\frac{v}{D})}{\alpha_{rD-r+j}+j}\right)\geq 0,\; 
v_{p_j}\left(\frac{B_{\alpha_{rD-r+j}+j}(\frac{v}{D})}{\alpha_{rD-r+j}+j}\right)=-1,\;1\leq j,\ell \leq r,\; j\neq \ell,\;1\leq v\leq D-1.
\end{equation}
From \eqref{pista}, using the basic properties of determinants and \eqref{berpe}, it follows that
\begin{equation}\label{stufz}
\Delta_{r,D}(\alf):= \begin{vmatrix}
\frac{{\tilde B}_{\alpha_1} (\frac{1}{D})}{\alpha_1} & \cdots & \frac{\tilde B_{\alpha_1}(\frac{D-1}{D})}{\alpha_1} & \frac{D^{\alpha_1}B_{\alpha_1}}{\alpha_1} 
& \cdots & \frac{\tilde B_{\alpha_1+r-1}(\frac{1}{D})}{\alpha_1+r-1} & \cdots & \frac{D^{\alpha_1}B_{\alpha_1+r-1}}{\alpha_1+r-1} \\
\frac{\tilde B_{\alpha_2}(\frac{1}{D})}{\alpha_2} & \cdots & \frac{\tilde B_{\alpha_2}(\frac{D-1}{D})}{\alpha_2} & \frac{D^{\alpha_2}B_{\alpha_2}}{\alpha_2} 
& \cdots & \frac{\tilde B_{\alpha_2+r-1}(\frac{1}{D})}{\alpha_2+r-1} & \cdots & \frac{D^{\alpha_2}B_{\alpha_2+r-1}}{\alpha_2+r-1} \\
\vdots & \vdots & \vdots & \vdots & \vdots & \vdots & \vdots  \\
\frac{\tilde B_{\alpha_{rD}}(\frac{1}{D})}{\alpha_{rD}} & \cdots & \frac{\tilde B_{\alpha_{rD}}(\frac{D-1}{D})}{\alpha_{rD}} & \frac{D^{\alpha_{rD}}B_{\alpha_{rD}}}{\alpha_{rD}} 
& \cdots & \frac{\tilde B_{\alpha_{rD}+r-1}(\frac{1}{D})}{\alpha_{rD}+r-1} & \cdots & \frac{D^{\alpha_{rD}}B_{\alpha_{rD}+r-1}}{\alpha_{rD}+r-1} 
\end{vmatrix} 
\end{equation}

\begin{prop}
 With the above assumptions, $\Delta_{r,D}(\alf)\neq 0$ if and only if $\tilde{\Delta}_{r,D}(\alf)\neq 0$.
\end{prop}

\begin{proof}
The conclusion follows from \eqref{pistaa}, \eqref{stufu}, \eqref{stuff}, \eqref{stufy} and \eqref{stufz}, using a similar argument as in the proof of Proposition $3.1$.
\end{proof}

\begin{prop}(Case $D=2$)
 With the above assumptions, $\Delta_{r,2}(\alf)\neq 0$.
\end{prop}

\begin{proof}
 By Proposition $3.2$, it is enough to prove that $\tilde \Delta_{r,2}(\alf)\neq 0$. We have that
\begin{equation}\label{suk} 
\tilde{\Delta}_{r,2}(\alf) = \begin{vmatrix}
\frac{{\tilde B}_{\alpha_1} (\frac{1}{2})}{\alpha_1} & \frac{{\tilde B}_{\alpha_1+1} (\frac{1}{2})}{\alpha_1+1} & \cdots &  \frac{{\tilde B}_{\alpha_1+r-1} (\frac{1}{2})}{\alpha_1+r-1} \\
\frac{{\tilde B}_{\alpha_2} (\frac{1}{2})}{\alpha_2} & \frac{{\tilde B}_{\alpha_2+1} (\frac{1}{2})}{\alpha_2+1} & \cdots &  \frac{{\tilde B}_{\alpha_2+r-1} (\frac{1}{2})}{\alpha_2+r-1} \\
 \vdots & \vdots & \vdots & \vdots \\
\frac{{\tilde B}_{\alpha_r} (\frac{1}{2})}{\alpha_r} & \frac{{\tilde B}_{\alpha_r+1} (\frac{1}{2})}{\alpha_1+1} & \cdots &  \frac{{\tilde B}_{\alpha_r+r-1} (\frac{1}{2})}{\alpha_r+r-1} \\
\end{vmatrix}.
\end{equation}
We choose $\alpha_j:=2^{j+t}-j+1$, where $2^t\geq r$. From \eqref{int} and Lemma $2.1(1)$ it follows that
\begin{equation}\label{suku}
v_2({\tilde B}_{\alpha_j+j-1}(\frac{1}{2}))=0,\; v_2({\tilde B}_{\alpha_j+\ell-1}(\frac{1}{2}))\geq 0,\;(\forall)1\leq j,\ell\leq r,\;j\neq \ell.  
\end{equation}
On the other hand, 
\begin{equation}\label{sukuu}
j+t = v_2(\alpha_j+j-1) > v_2(\alpha_j+\ell-1),\;(\forall)1\leq j,\ell\leq r,\;j\neq \ell.
\end{equation}
From \eqref{suk}, \eqref{suku} and \eqref{sukuu}, it follows that
$$ v_2(\tilde{\Delta}_{r,2}(\alf)) = v_2\left(\prod_{j=1}^r \frac{{\tilde B}_{\alpha_{j}+j-1} (\frac{1}{2})}{\alpha_j+j-1} \right) = -rt-\binom{r}{2} < \infty, $$
hence $\tilde \Delta_{r,2}(\alf)\neq 0$, as required.
\end{proof}

In the following, we assume $D\geq 3$. Let $N:=\lfloor \frac{(D-1)r}{2} \rfloor$.
We also assume that 
$\alpha_t$ is odd for all $1\leq t\leq N$,  and $\alpha_t$ is even for all $N +1 \leq t \leq r(D-1)$.
Let $k:= \lfloor \frac{D-1}{2} \rfloor$ and $\bar k = \lceil \frac{D-1}{2} \rceil$. From \eqref{unu} and \eqref{berpe} it follows that
\begin{equation}\label{pistaq}
  \tilde B_{\alpha_t+j-1}(\frac{D-v}{D}) + \tilde B_{\alpha_t+j-1}(\frac{v}{D}) = \begin{cases} 0,& \alpha_t+j-1\text{ is odd} \\
 2\tilde B_{\alpha_t+j-1}(\frac{v}{D}),& \alpha_t+j-1 \text{ is even} \end{cases},
\end{equation}
for all $1\leq t\leq r(D-1),\;1\leq v\leq \bar k$ and $1\leq j\leq r$. We consider the determinants:
\begin{equation}\label{pistap}
\tilde{\Delta'}_{r,D}(\alf):= \begin{vmatrix}
\frac{\tilde B_{\alpha_1}(\frac{1}{D})}{\alpha_1} & \cdots & \frac{\tilde B_{\alpha_1}(\frac{k}{D})}{\alpha_1}  
& \frac{\tilde B_{\alpha_1+1}(\frac{1}{D})}{\alpha_1+1} & \cdots & \frac{\tilde B_{\alpha_1+1}(\frac{\bar k}{D})}{\alpha_1+1} & \cdots \\
\frac{\tilde B_{\alpha_2}(\frac{1}{D})}{\alpha_2} & \cdots & \frac{\tilde B_{\alpha_2}(\frac{k}{D})}{\alpha_2}
& \frac{\tilde B_{\alpha_2+1}(\frac{1}{D})}{\alpha_2+1} & \cdots & \frac{\tilde B_{\alpha_2+1}(\frac{\bar k}{D})}{\alpha_2+1}& \cdots \\
\vdots & \vdots & \vdots & \vdots & \vdots & \vdots & \vdots  \\
\frac{\tilde B_{\alpha_{N}}(\frac{1}{D})}{\alpha_{N}} & \cdots & \frac{\tilde B_{\alpha_{N}}(\frac{k}{D})}{\alpha_{N}} 
& \frac{\tilde B_{\alpha_{N+1}}(\frac{1}{D})}{\alpha_{N+1}} & \cdots & \frac{\tilde B_{\alpha_{N+1}}(\frac{\bar k}{D})}{\alpha_{N+1}} & \cdots 
\end{vmatrix} \end{equation}
\begin{equation}\label{pistai}
\tilde{\Delta''}_{r,D}(\alf):= \begin{vmatrix}
\frac{\tilde B_{\alpha_{N+1}}(\frac{1}{D})}{\alpha_{N+1}} & \cdots & \frac{\tilde B_{\alpha_{N+1}}(\frac{\bar k}{D})}{\alpha_{N+1}}  
& \frac{\tilde B_{\alpha_{N+1}+1}(\frac{1}{D})}{\alpha_{N+1}+1} & \cdots & \frac{\tilde B_{\alpha_{N+1}+1}(\frac{k}{D})}{\alpha_{N+1}+1} & \cdots \\
\frac{\tilde B_{\alpha_{N+2}}(\frac{1}{D})}{\alpha_{N+2}} & \cdots & \frac{\tilde B_{\alpha_{N+2}}(\frac{\bar k}{D})}{\alpha_{N+2}}
& \frac{\tilde B_{\alpha_{N+2}+1}(\frac{1}{D})}{\alpha_{N+2}+1} & \cdots & \frac{\tilde B_{\alpha_{N+2}+1}(\frac{k}{D})}{\alpha_{N+2}+1}& \cdots \\
\vdots & \vdots & \vdots & \vdots & \vdots & \vdots & \vdots  \\
\frac{\tilde B_{\alpha_{rD-r}}(\frac{1}{D})}{\alpha_{rD-r}} & \cdots & \frac{\tilde B_{\alpha_{rD-r}}(\frac{\bar k}{D})}{\alpha_{rD-r}} 
& \frac{\tilde B_{\alpha_{rD-r}+1}(\frac{1}{D})}{\alpha_{rD-r}+1} & \cdots & \frac{\tilde B_{\alpha_{rD-r}+1}(\frac{k}{D})}{\alpha_{rD-r}+1} & \cdots 
\end{vmatrix} \end{equation}

\begin{prop}
 With the above assumptions, it holds that
$$\tilde{\Delta}_{r,D}(\alf) = C\tilde{\Delta'}_{r,D}(\alf) \tilde{\Delta''}_{r,D}(\alf),$$
 where $C\neq 0$.  In particular, if $\tilde{\Delta'}_{r,D}(\alf)\neq 0$ and $\tilde{\Delta''}_{r,D}(\alf)\neq 0$ then $\tilde{\Delta}_{r,D}(\alf)\neq 0$.
\end{prop}

\begin{proof}
In \eqref{pistaa}, we add the $(j+tr)$-th column over the $(D-j+tr)$-th column, where $1\leq j\leq k$ and $0\leq t\leq r-1$.
The conclusion follows from \eqref{pistaq}, \eqref{pistap} and \eqref{pistai} using the basic properties of determinants.
\end{proof}

\section{Proof of Theorem 1.2}

The case $D=1$ was proved in Proposition $3.1$. Also, the case $D=2$ was proved in Proposition $3.3$.
Assume that $D:=p>2$ is a prime number. Let $k:=\lfloor \frac{p-1}{2} \rfloor$. According to Proposition $3.4$,
it is enough to prove that $\tilde{\Delta'}_{r,p}(\alf)\neq 0$ and $\tilde{\Delta''}_{r,p}(\alf)\neq 0$.
Let 
\begin{equation}\label{40}
\log_p(r-1) < t_1 < t_2 < \cdots <t_r,
\end{equation}
 be a sequence of positive integers. We define
\begin{equation}\label{41}
\alpha_{j+(s-1)k}:=\begin{cases} 2jp^{t_s}-s+1,& s \text{ is even} \\ (2j-1)p^{t_s}-s+1,& s \text{ is odd} \end{cases},\;(\forall)1\leq s\leq r,\; 1\leq j\leq k.
\end{equation}
From \eqref{40} and \eqref{41} it follows that
\begin{eqnarray}
& v_p(\alpha_{j+(s-1)k}+s-1)=t_s,\;(\forall)1\leq s\leq r,\;1\leq j\leq k. \label{43} \\
& v_p(\alpha_{j+(s-1)k}+\ell)<t_1,\;(\forall)1\leq s\leq r,\;1\leq j\leq k \text{ and } 0\leq \ell\leq r-1 \text{ with }\ell\neq s-1.\label{44}
\end{eqnarray} 
On the other hand, from Lemma $2.1(2)$ it follows that 
\begin{equation}\label{45}
B_{\alpha_j}(\frac{v}{p})\equiv v^{\alpha_j}(\bmod\;p),\; (\forall)1\leq j\leq rp.
\end{equation}
From \eqref{43}, \eqref{44} and \eqref{45} it follows that
\begin{eqnarray}
& v_p(\frac{\widetilde B_{\alpha_{j+(s-1)k}+s-1}(\frac{v}{p})}{\alpha_{j+(s-1)k}+s-1} ) = -t_s,
\;(\forall)1\leq s\leq r,\;1\leq j,v \leq k, \label{46} \\
& v_p(\frac{\widetilde B_{\alpha_{j+(s-1)k}+\ell}(\frac{v}{p})}{\alpha_{j+(s-1)k}+\ell} ) > -t_1
\;(\forall)1\leq s\leq r,\;1\leq j,v\leq k \text{ and } 0\leq \ell\leq r-1 \text{ with }\ell\neq s-1.\label{47}
\end{eqnarray}
We consider the determinants
\begin{eqnarray}
& M_s : = \det\left(\widetilde B_{\alpha_{j+(s-1)k}+s-1}(\frac{v}{p})  \right)_{1\leq j,v\leq k},\; 1\leq s\leq r. \label{48}
\end{eqnarray}
From \eqref{45} it follows that 
\begin{eqnarray}
&  M_s \equiv \det\left(v^{2jp^{t_s}} \right)_{1\leq j,v\leq k} \equiv \det\left(v^{2j} \right)_{1\leq j,v\leq k} (\bmod\; p) \text{ for }s\text{ even}, \label{49} \\
&  M_s \equiv \det\left(v^{2(j-1)p^{t_s}} \right)_{1\leq j,v\leq k}\equiv \det\left(v^{2j-1} \right)_{1\leq j,v\leq k} (\bmod\; p) \text{ for }s\text{ odd}.\label{410}
\end{eqnarray}
On the other hand, using the Vandermonde formula, we have
\begin{eqnarray}
& \det\left(v^{2j} \right)_{1\leq j,v\leq k} = v^2 \prod_{1\leq i<j \leq k} (j-i)(j+i) \not \equiv 0 (\bmod p),\label{411}\\
& \det\left(v^{2j-1} \right)_{1\leq j,v\leq k} = v \prod_{1\leq i<j \leq k} (j-i)(j+i) \not \equiv 0 (\bmod p).\label{412}
\end{eqnarray}
From \eqref{49},\eqref{410},\eqref{411} and \eqref{412} it follows that 
\begin{equation}\label{413}
v_p(M_s)=0,\;(\forall)1\leq s\leq r, 
\end{equation}
hence, in particular $M_s\neq 0$. From \eqref{pistap}, \eqref{46}, \eqref{47}, \eqref{48} and \eqref{413} it follows that
$$v_p(\tilde{\Delta'}_{r,p}(\alf)) = -(t_1+t_2+\cdots+t_r)k,$$
hence, in particular, $\tilde{\Delta'}_{r,p}(\alf)\neq 0$. Similarly, one can prove that $\tilde{\Delta''}_{r,p}(\alf)\neq 0$.

\vspace{10pt}
\noindent
\textbf{Aknowledgment}: I would like to express my gratitude to Florin Nicolae for the valuable discussions regarding this paper.

{}

\vspace{2mm} \noindent {\footnotesize
\begin{minipage}[b]{15cm}
Mircea Cimpoea\c s, Simion Stoilow Institute of Mathematics, Research unit 5, P.O.Box 1-764,\\
Bucharest 014700, Romania, E-mail: mircea.cimpoeas@imar.ro
\end{minipage}}
\end{document}